\documentclass[onecolumn]{IEEEtran}
\usepackage[cmex10]{amsmath}
\usepackage{amsmath,amssymb,amsfonts,amsthm,mathrsfs}
\usepackage{cite}
\usepackage{graphicx}
\usepackage{subfigure}
 \newtheorem{theorem}{Theorem}[section]
 
 \newtheorem{proposition}[theorem]{Proposition}
 \newtheorem{lemma}[theorem]{Lemma}
  \newtheorem{definition}{Definition}[section]
 
  \newtheorem{assumption}{Assumption}
 \theoremstyle{remark}
 \newtheorem{remark}{Remark}
\ifCLASSINFOpdf
\else
\fi
\begin{document}
\title{Convergence of  Bregman Alternating Direction Method with
Multipliers \\ for Nonconvex Composite Problems}

\author{Fenghui~Wang, Zongben~Xu$^*$,
        and~Hong-Kun~Xu
\thanks{$^*$corresponding author.}
\thanks{F. Wang is with School of Mathematics and Statistics, Xian Jiaotong
University, Xian 710049, P R China}% <-this % stops a space
\thanks{Z. Xu is with School of Mathematics and Statistics, Xian Jiaotong
University, Xian 710049, P R China.}% <-this % stops a space
\thanks{H.K. Xu is with Department of Applied Mathematics, National Sun
Yat-sen University, Kaohsiung 80424, Taiwan.}}% <-this % stops a space

%\markboth{Journal of \LaTeX\ Class Files,~Vol.~11, No.~4, December~2012}%
%{Shell \MakeLowercase{\textit{et al.}}: Bare Demo of IEEEtran.cls for Journals}

\maketitle

\begin{abstract}
The alternating direction method with multipliers (ADMM) has been
one of most powerful and successful methods for solving  various
convex or nonconvex composite problems that arise in the fields of
image \& signal processing and machine learning. In convex settings,
numerous convergence results have been established for ADMM as well
as its varieties. However, there have been few studies on the
convergence properties of ADMM under nonconvex frameworks, since the
convergence analysis of nonconvex algorithm is generally very difficult.
In this paper we study the Bregman modification of ADMM (BADMM),
which includes the conventional ADMM as a special case and can significantly
 improve the performance of  the algorithm. Under some
assumptions, we show that the iterative sequence generated by BADMM
converges to a stationary point of the associated augmented
Lagrangian  function. The obtained results underline the feasibility
of ADMM in applications under nonconvex settings.
\end{abstract}


\IEEEpeerreviewmaketitle

\section{Introduction}
Many problems arising in the fields of  signal \& image processing
and machine learning \cite{bpc,yl} involve finding a minimizer of
some composite objective functions. More specifically, such problems
can be formulated as:
\begin{align}\label{p1}
\min & \ f(x)+g(y)\nonumber\\
 \mathrm{s.t.} & \  Ax=By,
\end{align}
where $A\in\mathbb{R}^{m\times n_1}$ and $B\in\mathbb{R}^{m\times
n_2}$ are
 given matrices,  $f:\mathbb{R}^{n_1}\to\mathbb{R}$ is usually a
 (quadratic, or logistic)  loss function, and
$g:\mathbb{R}^{n_2}\to\mathbb{R}$ is often a regularizer such as the
$\ell_1$ norm or $\ell_{1/2}$ quasi-norm.

Because of  its separable structure, problem (\ref{p1}) can be
efficiently solved by the alternating direction method with
multipliers (ADMM), which decomposes  the original joint
minimization problem  into two  easily solved subproblems. The
standard ADMM for problem (\ref{p1}) takes the form:
 \begin{align}
y^{k+1}&=\arg\min\limits_{y\in\mathbb{R}^{n_2}} L_{\alpha}(x^{k},y,p^k)  \label{M1}\\
x^{k+1}&=\arg\min\limits_{x\in\mathbb{R}^{n_1}} L_{\alpha}(x,y^{k+1},p^k)  \label{M2}\\
p^{k+1}&=p^k+\alpha(Ax^{k+1}-By^{k+1}), \label{M3}
\end{align}
where $\alpha$ is a penalty parameter and
\setlength{\arraycolsep}{0.0em}
\begin{align*}
L_{\alpha}(x,y,p)&:=f(x)+g(y)+\langle p,
Ax-By\rangle\\
&\qquad +\frac\alpha2\|Ax-By\|^2
\end{align*}
is the associated augmented Lagrangian function with multiplier $p$.
Generally speaking, ADMM is first minimized with respect to $y$ for
fixed values of $p, x$, then with respect to $x$ with $p,y$ fixed,
and finally maximized with respect to $p$ with $x, y$ fixed.
Updating the dual variable  $p^k$  in the above system is a trivial
task, but this is not so simple for the primal variables  $x^k$ and
$y^k$. Indeed  in many cases, the $x$-subproblem (\ref{M2}) and
$y$-subproblem (\ref{M1}) cannot easily be solved.  Recently, the
Bregman modification of ADMM (BADMM) has been adopted by several
researchers to improve the performance of the conventional ADMM
algorithm  \cite{fwb,wb,wang,zhang}.  BADMM takes the following iterative
form: \footnote{If the solution to the $x$ or $y$-subproblem is not
unique, then $x^k$ or $y^k$ should be regarded as a selection from
their solution sets. } \setlength{\arraycolsep}{0.0em}
\begin{align}
y^{k+1}&=\arg\min\limits_{y\in\mathbb{R}^{n_2}} L_{\alpha}(x^{k},y,p^k)+\triangle_{\psi}(y,y^k) \label{A3}\\
x^{k+1}&=\arg\min\limits_{x\in\mathbb{R}^{n_1}} L_{\alpha}(x,y^{k+1},p^k)+\triangle_{\phi}(x,x^k) \label{A2}\\
p^{k+1}&=p^k+\alpha(Ax^{k+1}-By^{k+1}),\label{A4}
\end{align}
where $\triangle_{\psi}$ and $\triangle_{\phi}$ respectively denote
the Bregman distance with respect to function $\psi$ and $\phi.$ The
difference between this algorithm and the standard ADMM is that the
objective function in (\ref{M1})-(\ref{M2}) is replaced by the sum
of a Bregman distance function and the augmented Lagrangian
function. Moreover, as shown in \cite{wang,zhang,lp} and the
following section, an appropriate choice of Bregman distance does
indeed simplify the original subproblems.

 ADMM was introduced in the early 1970s \cite{gm,gm2},  and its convergence
  properties for convex objective functions have been extensively studied.
  The convergence of ADMM was first established
for strongly convex functions \cite{gm,gm2}, before being extended
to general convex functions \cite{eck,ecb}.  It has been shown that
ADMM converges at a sublinear rate of  $\mathcal{O}(1/k)$
\cite{hy,ms}, or $\mathcal{O}(1/k^2)$ for the accelerated version
\cite{gds}; furthermore, a linear convergence rate was also shown
under certain additional assumptions \cite{dy}. The convergence of
BADMM for convex objective functions has also been examined with the
Euclidean distance \cite{ct}, Mahalanobis distance \cite{zhang}, and
the general Bregman distance \cite{zhang}.

Recent studies on nonnegative matrix factorization, distributed
matrix factorization, distributed clustering, sparse zero variance
discriminant analysis, polynomial optimization, tensor
decomposition, and matrix completion have led to growing interest in
ADMM for nonconvex objective functions (see e.g.
\cite{hcw,ls,xyw,zk,zy}). It has been shown that the nonconvex ADMM
works extremely well for these particular examples.

However, because the convergence analysis of nonconvex algorithms is
generally very difficult, there have been few studies on the
convergence properties of ADMM under nonconvex frameworks. One major
difficulty is that the F\'{e}jer monotonicity of iterative sequences
does not hold in the absence of convexity. Very recently, \cite{hlr}
analyzed the convergence of ADMM for certain nonconvex consensus and
sharing problems. They demonstrated that   with $A$ and $B$  set to
the identity matrices, ADMM converges to {\em the set of stationary
solutions} as long as the penalty parameter $\alpha$ is sufficiently
large. To show the convergence of ADMM to {\em a stationary point},
additional assumptions are required on the functions involved. For
example, if $f$ and $g$ are both semi-algebraic,  \cite{lp} proved
that ADMM converges to a stationary point when $B$ is the identity
matrix. This result requires that  function $f$ is strongly convex
or matrix $A$ has full-column rank.

In this paper, we study the convergencev of BADMM under nonconvex
frameworks.  First, we extend the convergence of the BADMM from
semi-algebraic functions to sub-analytic functions.  In particular,
this implies that BADMM is convergent for logistic sparse loss functions,
which are not semi-algebraic.  Second, we establish a global
convergence theorem for cases when $B$ has full-column rank. This
allows us to choose $\phi\equiv0$, which covers a recent result in
\cite{lp}. We also study the case when $B$ does not have full-column
rank.  In this instance, a suitable Bregman distance also leads to
global BADMM convergence. This enhanced flexibility of BADMM enables
its application to more general cases. More importantly, the main
idea of our convergence analysis is different from that used in
\cite{lp}. Instead of employing an augmented Lagrangian function at
each iteration, we demonstrate global convergence using the descent
property of an auxiliary function.

The paper is organized as follows. In Section 2, we recall the
definitions of subdifferentials, Bregman distance, and
Kurdyka-{\L}ojasiewicz inequality. In Section 3, we establish the
global convergence of BADMM to a critical point
 under  certain assumptions. In Section 4, we
 conduct  experimental studies to
verify the convergence of BADMM.

\section{Preliminaries}

In what follows, $\mathbb{R}^n$ will stand for the $n$-dimensional
Euclidean space,
$$\langle x,y\rangle=x^{\top}y=\sum _{i=1}^n
x_iy_i, \ \|x\|=\sqrt{\langle x, x\rangle},$$ where
$x,y\in\mathbb{R}^n$ and $\top$ stands for the transpose operation.

\subsection{Subdifferentials}

Given a function $f:\mathbb{R}^n\to \mathbb{R}$ we denote by
$\mathrm{dom}f$ the domain of $f$, namely $\mathrm{dom}f:=\{x\in
\mathbb{R}^n: f(x)<+\infty \}$. A function $f$ is said to be proper
if $\mathrm{dom}f\neq\emptyset;$ lower  semicontinuous at the point
$x_0$ if
$$\liminf_{x\to x_0}f(x)\ge f(x_0).$$
 If $f$ is lower semicontinuous at every point of its domain of definition, then it is simply
 called a lower semicontinuous function.

\begin{definition}
Let $f:\mathbb{R}^n\to \mathbb{R}$ be a proper lower semi-continuous
function.
\begin{itemize}
  \item[(i)] Given $x\in \mathrm{dom} f,$ the Fr\'{e}chet subdifferential of $f$ at $x$,
  written by $\widehat{\partial} f(x)$, is the set of  all elements $u\in \mathbb{R}^n$ which satisfy
  \begin{align*}
\lim_{y\neq x}\inf_{y\to x}\frac{f(y)-f(x)-\langle u,
y-x\rangle}{\|x-y\|}\ge0.
  \end{align*}
  \item[(ii)] The limiting  subdifferential, or simply subdifferential, of $f$ at $x$,
  written by $\partial f(x)$, is defined as
  \begin{align*}
\partial f(x)=\{u\in \mathbb{R}^n: \exists x^k\to x, f(x^k)\to f(x),\\
u^k\in\widehat{\partial} f(x^k)\to u, k\to\infty\}.
  \end{align*}
\item[(iii)]
A critical point or stationary point of $f$ is a point $x_0$ in the
domain of $f$ satisfying $0\in\partial f(x_0).$\end{itemize}
\end{definition}

\begin{definition}
An element $z^*:=(x^{*}, y^*, p^*)$ is  called a critical point or
stationary point of the Lagrangian function $L_{\alpha}$ if it
satisfies:
\begin{align}
\left\{
  \begin{array}{ll}
 -A^{\top}p^*=\nabla f(x^{*}) \\
 B^{\top}p^*\in \partial g(y^*) \\
 Ax^*=By^*.
  \end{array}
\right.
\end{align}
\end{definition}
Let us now collect some basic properties of the subdifferential (see
\cite{mor}).

\begin{proposition}
Let $f:\mathbb{R}^n\to \mathbb{R}$ and $g:\mathbb{R}^n\to
\mathbb{R}$ be proper lower semi-continuous functions.
\begin{itemize}
  \item   $\widehat{\partial} f(x)\subset  \partial  f(x)$ for each $x\in \mathbb{R}^n.$
  Moreover, the first set is closed and convex, while the second is closed, and not necessarily convex.
  \item  Let $(u^k, x^k)$ be sequences such that $x^k\to x, u^k\to u, f(x^k)\to f(x)$ and $u^k\in \partial f(x^k).$
  Then by the definition of the subdifferential, we have $u\in \partial f(x).$
  \item The Fermat's rule remains true: if $x_0\in\mathbb{R}^n$ is a local minimizer of $f$, then $x_0$ is a
  critical point or stationary point of $f$, that is, $0\in \partial f(x_0).$
  \item If $f$ is continuously differentiable function, then $\partial (f+g)(x)=\nabla f(x)+ \partial g(x).$
\end{itemize}
\end{proposition}

 A function $f$ is said to be
 {\em  $\ell_{f}$-Lipschitz continuous}  $(\ell_{f}\ge 0)$ if
 $$\|f(x)-f(y)\|\le\ell_{f}\|x-y\|$$
 for any $x,y\in\mathrm{dom}f$; {\em  $\mu$-strongly convex} $(\mu>0)$ if
 \begin{align}\label{B1}
 f(y)\ge f(x)+\langle\xi(x), y-x\rangle+\frac\mu2\|y-x\|^2,
 \end{align}
 for any $x,y\in\mathrm{dom}f$ and $\xi(x)\in\partial f(x).$

\subsection{Kurdyka-{\L}ojasiewicz inequality}\label{s1}

The Kurdyka-{\L}ojasiewicz (K-L) inequality plays an important role
in our subsequent analysis. This inequality was first introduced
by {\L}ojasiewicz \cite{loj} for real real analytic functions, and
then was extended by Kurdyka \cite{kur} to smooth functions whose
graph belongs to an o-minimal structure, and recently was further
extended to  nonsmooth sub-analytic functions \cite{bdo}.

\begin{definition}[K-L inequality]
A function $f:\mathbb{R}^n\to \mathbb{R}$ is said to satisfy the
 K-L inequality at $x_0$ if
 there exists $\eta>0, \delta>0, \varphi\in\mathscr{A}_{\eta}$, such that for all
$x\in\mathcal{O}(x_0,\delta)\cap\{x: f(x_0)<f(x)<f(x_0)+\eta\}$
\begin{align*}
 \varphi'(f(x)-f(x_0))\mathrm{dist}(0, \partial f(x))\ge1,
\end{align*}
where $\mathrm{dist}(x_0, \partial f(x)):=\inf\{\|x_0-y\|:y\in
\partial f(x)\},$ and
 $\mathscr{A}_{\eta}$ stand for the class of functions $\varphi:[0,\eta)\to
\mathbb{R}^{+}$ such that (a) $\varphi$ is  continuous on
$[0,\eta)$; (b) $\varphi$ is smooth concave on $(0,\eta)$; (c)
 $\varphi(0)=0, \varphi'(x)>0, \forall x\in (0,\eta)$.
\end{definition}

The following is an extension of the conventional K-L inequality
\cite{bst}.

\begin{lemma}[K-L inequality on compact subsets]\label{KL}
Let $f:\mathbb{R}^n\to \mathbb{R}$ be a proper lower semi-continuous
function and let $\Omega\subseteq \mathbb{R}^n$ be a compact set. If
$f$ is a constant on $\Omega$ and $f$ satisfies the K-L inequality
at each point in $\Omega$, then
 there exists $\eta>0, \delta>0, \varphi\in\mathscr{A}_{\eta}$, such that  for all $x_0\in\Omega$ and for all
$x\in\{x\in\mathbb{R}^n:\mathrm{dist}(x,\Omega)<\delta)\}\cap\{x\in\mathbb{R}^n:
f(x_0)<f(x)<f(x_0)+\eta\}$,
\begin{align*}
 \varphi'(f(x)-f(x_0))\mathrm{dist}(0, \partial f(x))\ge1.
\end{align*}
\end{lemma}

Typical functions satisfying the K-L inequality include strongly
convex functions, real analytic functions, semi-algebraic functions
and sub-analytic functions.

A subset $C\subset \mathbb{R}^n$ is said to be {\em semi-algebraic}
if it can be written as
  \begin{align*}
   C=\bigcup_{j=1}^r\bigcap_{i=1}^{s}\{x\in \mathbb{R}^n: g_{i,j}(x)=0, h_{i,j}(x)<0\},
  \end{align*}
  where $g_{i,j},h_{i,j}:\mathbb{R}^n\to\mathbb{R}$ are real polynomial functions.
Then a function $f:\mathbb{R}^n\to \mathbb{R}$ is called {\em
semi-algebraic}  if its graph
  \begin{align*}
  \mathrm{G}(f):=\{(x,y)\in\mathbb{R}^{n+1}:f(x)=y\}
  \end{align*}
  is a semi-algebraic subset in $\mathbb{R}^{n+1}$. For example, the $\ell_{q}$ quasi norm
  $\|x\|_{q}:=(\sum_i|x_i|^{q})^{1/q}$ with $0<q\le1$, the sup-norm $\|x\|_{\infty}:=\max_i|x_i|,$
  the Euclidean norm $\|x\|$, $\|Ax-b\|^{q}_{q}$,
  $\|Ax-b\|$ and $\|Ax-b\|_\infty$ are all semi-algebraic functions \cite{bst,yin}.

A real function on  $\mathbb{R}$ is said to be {\em analytic} if it
possesses derivatives of all orders
 and agrees with its Taylor series in a neighborhood of every point. For a real function $f$  on
 $\mathbb{R}^n$, it is said to be {\em analytic} if the function of one variable
 $g(t):=f(x+ty)$ is analytic for any $x,y\in\mathbb{R}^n$.
It is readily seen that  real polynomial functions such as quadratic
functions
  $\|Ax-b\|^2$  are analytic. Moreover the $\varepsilon$-smoothed $\ell_{q}$ norm
  $\|x\|_{\varepsilon,q}:=\sum_i(x_i^2+\varepsilon)^{q/2}$ with $0<q\le1$ and the logistic
  loss function $\log(1+e^{-t})$ are also examples for real analytic functions \cite{yin}.

A subset $C\subset \mathbb{R}^n$ is said to be {\em sub-analytic} if
it can be written as
  \begin{align*}
   C=\bigcup_{j=1}^r\bigcap_{i=1}^{s}\{x\in \mathbb{R}^n: g_{i,j}(x)=0, h_{i,j}(x)<0\},
  \end{align*}
  where $g_{i,j},h_{i,j}:\mathbb{R}^n\to\mathbb{R}$ are real analytic functions.
Then a function $f:\mathbb{R}^n\to \mathbb{R}$ is called {\em
sub-analytic} if its graph $\mathrm{G}(f)$ is a sub-analytic subset
in $\mathbb{R}^{n+1}$. It is clear that both real analytic and
semi-algebraic functions are sub-analytic. Generally speaking, the
sum of of two sub-analytic functions is not necessarily
sub-analytic. As shown in \cite{bdo,yin}, for two sub-analytic
functions, if at least one function maps bounded sets to bounded
sets, then their sum is also sub-analytic. In particular, the sum of
a sub-analytic function and a analytic function is sub-analytic.
Some sub-analytic functions that are widely used are as follows:
\begin{itemize}
\item $\|Ax-b\|^2+\lambda\|y\|^q_q $;
\item $\|Ax-b\|^2+\lambda\sum_i(y^2_i+\varepsilon)^{q/2}$;
  \item
  $\frac1n\sum_{i=1}^n\log(1+\exp(-c_i(a_i^{\top}x+b))+\lambda\|y\|^q_q$;
\item
  $\frac1n\sum_{i=1}^n\log(1+\exp(-c_i(a_i^{\top}x+b))+\lambda\sum_i(y^2_i+\varepsilon)^{q/2}$.
\end{itemize}

\subsection{Bregman distance}
The  Bregman distance, first introduced  in 1967 \cite{berg}, plays
an important role in various iterative algorithms. As a
generalization of squared Euclidean distance, the Bregman distance
share many similar nice properties of the Euclidean distance.
However, the Bregman distance is  not a metric, since it does not
satisfy the
 triangle inequality nor symmetry.
 For a convex differential function $\phi$, the associated Bregman distance is defined as
\begin{align*}
\triangle_{\phi}(x,y)=\phi(x)-\phi(y)-\langle\nabla\phi(y),x-y\rangle.
\end{align*}
In particular, if we let $\phi(x):=\|x\|^2$ in the above, then it is
reduced to $\|x-y\|^2$, namely the classical Euclidean distance.
Some nontrivial examples of Bregman distance include \cite{Ban}:
\begin{itemize}
\item Itakura-Saito distance:  $\sum_ix_i(\log
  x_i/y_i)- \sum_i(x_i-y_i)$;
  \item Kullback-Leibler divergence: $\sum_ix_i(\log
  x_i/y_i)$;
\item Mahalanobis distance: $\|x-y\|^2_Q=\langle Qx, x\rangle$
  with $Q$ a symmetric positive definite matrix.
\end{itemize}

Let us now collect some useful properties about Bregman distance.

 \begin{proposition}\label{pro}
Let $\phi$ be a convex differential function and
$\triangle_{\phi}(x,y)$ the associated  Bregman distance.
\begin{itemize}
  \item Non-negativity: $\triangle_{\phi}(x,y)\ge0,\triangle_{\phi}(x,x)=0$ for all $x,y$.
  \item Convexity: $\triangle_{\phi}(x,y)$ is convex in $x$, but not necessarily in $y$.
  \item Strong Convexity: If $\phi$ is $\delta$-strongly convex, then $\triangle_{\phi}(x,y)
  \ge\frac{\delta}{2}\|x-y\|^2$ for all $x,y$.
\end{itemize}
 \end{proposition}

 As shown in the below, an appropriate  choice of Bregman distance will
 simplify
 the $x$ and $y$-subproblems, which in turn  improve the performance of the
  algorithm. For example, in $y$-subproblem \eqref{A3}, when taking
 $g(y)=\|y\|^{1/2}_{1/2}, \psi\equiv0,$ then the problem is minimizing
 function
 \begin{align*}
\|y\|^{1/2}_{1/2}-\langle p^k, y\rangle+\frac\alpha2\|By-Ax^k\|^2.
 \end{align*}
In general finding a minimizer  of this function
 is not a easy task. However, if we take $\psi=\frac{\mu}{2}\|y\|^2-\frac\alpha2\|By-Ax^k-p^k/\alpha\|^2$
 with $\mu>\alpha\|B\|^2$, then it is transformed  into minimizing a
 problem of
 \begin{align*}
\|y\|^{1/2}_{1/2}+\frac{\alpha}{2\mu}\|y-(y^k-\mu^{-1}B^{\top}(By^k-Ax^k-p^k/\alpha))\|^2.
 \end{align*}
Such a problem has a closed form solution (see \cite{xc}), and thus
it can be very easily solved.

\subsection{Basic assumption}

We need the following basic assumptions on problem \eqref{p1}.
A basic assumption to guarantee the convergence of the BADMM
is that the matrix $A$ has full-row rank.  The
only difference between Assumptions \ref{j1} and \ref{j2} is: one
needs $B$ having full column rank in Assumption \ref{j1}, while in
Assumption \ref{j2} one needs $\psi$ being strongly convex. It worth
noting that one can choose $\psi\equiv0$ under Assumption \ref{j1},
so that the BADMM includes the standard ADMM
as a special case. It is also worth noting that  the choice of
$\psi\equiv0$ is not available under Assumption \ref{j2}.

\begin{assumption}\label{j1}
Let  $\min(\mu_0,\mu_1)>0,$  $f:\mathbb{R}^{n_1}\to \mathbb{R}$ a
continuous differential function and $g:\mathbb{R}^{n_2}\to \mathbb{R}$
a proper lower semi-continuous functions. Assume that the following hold.
\begin{itemize}
  \item[(a)] $AA^{\top}\succeq \mu_0I$ and $B$  is injective;
  \item[(b)]  either $L_{\alpha}(x,y,p)$  with respect to $x$
  or $\phi$ is $\mu_1$ strongly convex;
  \item[(c)] $f+g$ is a sub-analytic function, and $\nabla f,
  \nabla\phi$ and $\nabla\psi$ are Lipshitz continuous.
\end{itemize}
\end{assumption}

In condition (b), the strong convexity of $\phi$ is easily attained,
for example $\phi=\frac{\mu_1}{2}\|x\|^2,$ while the strong
convexity of $L_{\alpha}(x,y,p)$ in $x$ can be deduced from some
standard assumptions, for example Neumann boundary condition in
image processing \cite{ess}. Condition (b) will be used to guarantee
the sufficient descent property of the augmented Lagrangian
functions. More specifically, it implies
\begin{align}\label{B2}
L_{\alpha}(x^{k+1},y^{k+1},p^{k})&\le
L_{\alpha}(x^{k},y^{k+1},p^{k})-\frac{\mu_1}{2}\|x^{k+1}-x^k\|^2,
\end{align}
where $(x^k,y^k,p^k)$ is generated by algorithm
\eqref{A3}-\eqref{A4}. As a matter of fact, if $L_{\alpha}(x,y,p)$
with respect to $x$ is $\mu_1$-strongly convex, then
$L_{\alpha}(x,y,p)+\triangle_\phi$ is also $\mu_1$-strongly convex
because $\triangle_\phi$ is convex from Proposition \ref{pro}. Thus
the desired  inequality will follow from the definition of strong
convexity and Proposition \ref{pro}. If $\phi$ is strongly convex,
then it follows again from Proposition \ref{pro} that
\begin{align*}
\triangle_\phi(x^{k+1},x^k)\ge\frac{\mu_1}{2}\|x-x^k\|^2,
\end{align*}
which together with the definition of $x^k$ yields the desired
inequality.

The condition that $f+g$ is sub-analytic in (c) will be used to
guarantee the auxiliary function constructed in the following
section satisfying the K-L inequality. We  notice  that all
functions mentioned in subsection \ref{s1} satisfy  assumption (c).
The Lipschitz continuity is a standard assumption for various
algorithms, even in convex settings.

We also consider the BADMM under another set of conditions listed in Assumption
\ref{j2} below. The only difference between Assumptions \ref{j1} and \ref{j2}
is that one needs $B$ having full column rank in Assumption \ref{j1},
where in Assumption \ref{j2} we assume that $\psi$ is strongly convex. It is worth 
noting that one can choose $\psi\equiv0$ under Assumption \ref{j1}, so that
 the BADMM includes the standard ADMM as a special case.

\begin{assumption}\label{j2}
Let  $\min(\mu_0,\mu_1)>0,$  $f:\mathbb{R}^{n_1}\to \mathbb{R}$ a
continuous differential function and $g:\mathbb{R}^{n_2}\to \mathbb{R}$
a proper lower semi-continuous functions. Assume that the following hold.
\begin{itemize}
  \item[(a')] $AA^{\top}\succeq \mu_0I$ and $\psi$ is $\mu_2$-strongly convex.
  \item[(b)]  either $L_{\alpha}(x,y,p)$  with respect to $x$ or $\phi$ is $\mu_1$ strongly convex.
  \item[(c)] $f+g$ is a sub-analytic function, and $\nabla f, \nabla\phi$ and $\nabla\psi$ are Lipshitz continuous.
\end{itemize}
\end{assumption}

%%%%%%%%%%%%%%%%%%%%%%%%%%%%%%%%%%%%%%%%%%%%%%%%%%%%%%%%%%%%%%%%%%%%%%%%%%%%%%%%%%%%%
\section{Convergence Analysis}

In this section we  prove the convergence of BADMM under two
different assumptions. In both assumptions, the parameter $\alpha$
is chosen so that
\begin{align*}
 \alpha>\frac{4((\ell_{f}+\ell_{\phi})^2+\ell_{\phi}^2)}{\mu_1\mu_0},
\end{align*}
where $\ell_{f}$ and $\ell_{\phi}$ respectively stand for the
Lipshitz constant of functions $f$ and $\phi$.

According to a recent work \cite{Al96}, the key point for
convergence analysis of  nonconvex algorithms is to show the
descent property of the augmented Lagrangian function. This is however not
easily attained since the dual variable is updated by maximizing
the  augmented Lagrangian function. As an alternative way, we construct an
auxiliary function below, which helps us to deduce the global
convergence of BADMM.

%For our convenience we shall assume that $g$ is continuous
%in what follows, which  can be weakened to lower semi-continuality.

\subsection{The case $B$  is injective}

\begin{lemma}\label{E1}
Let Assumption \ref{j1} be fulfilled. Then there exists $\sigma_i>0,
i=0,1$ such that
\begin{align*}
\sigma_1\|x^{k+1}-x^k\|^2\le\hat{L}(x^{k},y^{k},p^{k},x^{k-1})
-\hat{L}(x^{k+1},y^{k+1},p^{k+1},x^k),
\end{align*}
 where
 $\hat{L}(x,y,p,\hat{x}):=L_\alpha(x,y,p)+\frac{\sigma_0}{2}
\|x-\hat{x}\|^2.$
\end{lemma}

\begin{proof}
First we show that for each $k\in\mathbb{N}$
\begin{align}\label{D1}
\|p^{k+1}-p^{k}\|^2&\le\frac{2(\ell_{f}+\ell_{\phi})^2}{\mu_0}
\|x^{k+1}-x^{k}\|^2\nonumber\\
&\quad+\frac{2\ell_{\phi}^2}{\mu_0}\|x^{k}-x^{k-1}\|^2.
\end{align}
Indeed applying Fermat's rule to \eqref{A2} yields
\begin{align}\label{D2}
\begin{split}
 \nabla f(x^{k+1})+A^{\top}p^k+\alpha A^{\top}(Ax^{k+1}-By^{k+1})\\
 +\nabla\phi(x^{k+1})-\nabla\phi(x^k)=0,
\end{split}
\end{align}
which together with \eqref{A4} implies that
\begin{align}\label{A1}
\begin{split}
 A^{\top}p^{k+1}&=A^{\top}(p^{k}+\alpha (Ax^{k+1}-By^{k+1}))\\
&=-\nabla f(x^{k+1})+\nabla\phi(x^k)-\nabla\phi(x^{k+1}).
\end{split}
\end{align}
It then follows that
\begin{align*}
&\quad \ \|A^{\top}(p^{k+1}-p^{k})\|^2\\
&=\|\nabla f(x^{k+1})-\nabla f(x^{k}) +(\nabla\phi(x^{k+1})\\
&\quad-\nabla\phi(x^k))+(\nabla\phi(x^{k-1})-\nabla\phi(x^k))\|^2\\
&\le\big(\|\nabla f(x^{k+1})-\nabla f(x^{k})\| +\|\nabla\phi(x^{k+1})\\
&\quad-\nabla\phi(x^k)\|+\|\nabla\phi(x^{k-1})-\nabla\phi(x^k)\|\big)^2\\
&\le\big(\ell_{f}\|x^{k+1}-x^{k}\|+\ell_{\phi}\|x^{k}-x^{k+1}\|\\
&\quad+\ell_{\phi}\|x^{k}-x^{k-1}\|\big)^2\\
&\le2(\ell_{f}+\ell_{\phi})^2\|x^{k+1}-x^{k}\|^2\\
&\quad+2\ell_{\phi}^2\|x^{k}-x^{k-1}\|^2.
\end{align*}
Since matrix $A$ is surjective, we have
\begin{align*}
\|A^{\top}(p^{k+1}-p^{k})\|^2&=\langle A^{\top}(p^{k+1}-p^{k}), A^{\top}(p^{k+1}-p^{k})\rangle\\
&=\langle AA^{\top}(p^{k+1}-p^{k}), p^{k+1}-p^{k}\rangle\\
&\ge\mu_0 \|p^{k+1}-p^{k}\|^2,
\end{align*}
which at once implies \eqref{D1}, as desired.

Next we claim that
\begin{align}\label{D4}
\begin{split}
&L_{\alpha}(x^{k+1},y^{k+1},p^{k+1})-
L_{\alpha}(x^{k},y^{k},p^{k})\\
&\le-\frac{\mu_1}{2}\|x^{k+1}-x^k\|^2
+\frac1\alpha\|p^{k+1}-p^{k}\|^2.
\end{split}
\end{align}
To see this, we deduce  from \eqref{B2} and \eqref{A3}-\eqref{A4}
that
\begin{align*}
&L_{\alpha}(x^{k},y^{k+1},p^{k})\le L_{\alpha}(x^{k},y^{k},p^{k}),\\
&L_{\alpha}(x^{k+1},y^{k+1},p^{k})\le L_{\alpha}(x^{k},y^{k+1},p^{k})\\
&\qquad\qquad\qquad\qquad-\frac{\mu_1}{2}\|x^{k+1}-x^k\|^2, \\
&L_{\alpha}(x^{k+1},y^{k+1},p^{k+1})-L_{\alpha}(x^{k+1},y^{k+1},p^{k})\\
&=\langle p^{k+1}-p^{k}, Ax^{k+1}-By^{k+1}\rangle\\
&=\frac1\alpha\|p^{k+1}-p^{k}\|^2.
\end{align*}
Adding up the above formulas at once yields \eqref{D4}.

Finally it follows from \eqref{D1} and \eqref{D4} that
\begin{align*}
&L_{\alpha}(x^{k+1},y^{k+1},p^{k+1})-L_{\alpha}(x^{k},y^{k},p^{k})\\
&\le\left(\frac{2(\ell_{f}+\ell_{\phi})^2}{\alpha\mu_0}-\frac{\mu_1}{2}\right)
\|x^{k+1}-x^{k}\|^2\\
&\quad+\frac{2\ell_{\phi}^2}{\alpha\mu_0}\|x^{k}-x^{k-1}\|^2,
\end{align*}
which is equivalent to
\begin{align*}
&L_{\alpha}(x^{k+1},y^{k+1},p^{k+1})+\frac{2\ell_{\phi}^2}{\alpha\mu_0}\|x^{k+1}-x^k\|^2\\
&\le L_{\alpha}(x^{k},y^{k},p^{k})+\frac{2\ell_{\phi}^2}{\alpha\mu_0}\|x^{k}-x^{k-1}\|^2\\
&\quad-\left(\frac{\mu_1}{2}-\frac{2(\ell_{f}+\ell_{\phi})^2}{\alpha\mu_0}-
  \frac{2\ell_{\phi}^2}{\alpha\mu_0}\right)\|x^{k}-x^{k+1}\|^2.
\end{align*}
Let  us now define
\begin{align*}
&\sigma_0=\frac{2\ell_{\phi}^2}{\alpha\mu_0}, \ \sigma_1=
\left(\frac{\mu_1}{2}-\frac{2(\ell_{f}+\ell_{\phi})^2}{\alpha\mu_0}-
  \frac{2\ell_{\phi}^2}{\alpha\mu_0}\right).
\end{align*}
Clearly both $\sigma_i$ are positive and thus the desired inequality
follows.
\end{proof}

\begin{lemma}\label{E2}
If the sequence $z^k:=(x^k,y^k,p^k)$ is bounded, then
$$\sum_{k=0}^{\infty}\|z^k-z^{k+1}\|^2<\infty.$$
In particular the sequence $\|z^k-z^{k+1}\|$ is asymptotically
regular, namely $\|z^k-z^{k+1}\|\to0$ as $k\to\infty.$ Moreover any
cluster point
 of $z^k$ is a stationary point of $L_{\alpha}.$
\end{lemma}

\begin{proof}
Let $\hat{z}^k:=(x^k,y^k,p^k,x^{k-1}).$ Since $\hat{z}^k$ is clearly
bounded, there exists a subsequence $\hat{z}^{k_j}$ so that it is
convergent to some element $\hat{z}^{*}$. By our hypothesis the
function $\hat{L}$ is lower semicontinuous, which leads to
$$\liminf_{j\to\infty}\hat{L}(\hat{z}^{k_j})\ge \hat{L}(\hat{z}^*),$$
so that $\hat{L}(\hat{z}^{k_j})$ is bounded from below. By the
previous lemma, $\hat{L}(\hat{z}^{k})$ is nonincreasing, so that
$\hat{L}(\hat{z}^{k_j})$ is convergent. Moreover
$\hat{L}(\hat{z}^{k})$ is also convergent and
$\hat{L}(\hat{z}^{k})\ge
 \hat{L}(\hat{z}^{*})$ for each $k$.

Now fix $k\in\mathbb{N}.$ It then follows from Lemma \ref{E1} that
\begin{align*}
&\sigma_1\sum_{i=0}^{k}\|x^i-x^{i+1}\|^2\\
&\le\sum_{i=0}^{k} \hat{L}(\hat{z}^i)-\hat{L}(\hat{z}^{i+1})\\
&=\hat{L}(\hat{z}^{0})-\hat{L}(\hat{z}^{k+1})\\
&\le\hat{L}(\hat{z}^{0})-\hat{L}(\hat{z}^{*})<\infty.
\end{align*}
Since $k$ is chosen arbitrarily, we have
$\sum_{k=0}^{\infty}\|x^k-x^{k+1}\|^2<\infty,$ which with \eqref{D1}
implies $\sum_{k=0}^{\infty}\|p^k-p^{k+1}\|^2<\infty$. Since $B$ is
injective, it is readily seen that there exists $\mu_B>0$ so that
\begin{align}\label{B3}
&\quad \ \alpha^2 \mu_B\|y^k-y^{k+1}\|^2\nonumber\\
&\le\|\alpha B(y^k-y^{k+1})\|^2 \nonumber\\
&=\|(p^k-p^{k+1})+(p^k-p^{k-1})\nonumber\\
&\quad+\alpha( Ax^{k+1}-Ax^{k})\|^2 \nonumber\\
&\le2(\|p^k-p^{k+1}\|^2+\|p^k-p^{k-1}\|^2\nonumber\\
&\quad+\alpha^2\|A\|^2\|x^{k+1}-x^{k}\|^2).
\end{align}
Hence $\sum_{k=0}^{\infty}\|y^k-y^{k+1}\|^2<\infty$, so that
$\sum_{k=0}^{\infty}\|z^k-z^{k+1}\|^2<\infty;$ in particular
$\|z^k-z^{k+1}\|\to0$.

Let $z^*=(x^*,y^*,p^*)$ be any cluster point of $z^k$ and let
$z^{k_j}$ be a subsequence of $z^k$  converging to $z^*$. Since
$\|z^k-z^{k+1}\|$ tends to zero as $k\to\infty,$ $z^{k_j}$ and
$z^{k_j+1}$ have the same limit point $z^*$. Since
$\hat{L}(\hat{z}^{k})$ is convergent, it is not hard to see that
$g(y^{k})$ is also convergent.
 It then  follows from \eqref{A3}-\eqref{A4}  that
\begin{align*}
p^{k+1}&=p^{k}+\alpha (Ax^{k+1}-By^{k+1}),\\
-\nabla f(x^{k+1})&  = A^{\top}p^{k+1}+\nabla\phi(x^{k+1})-\nabla\phi(x^k),\\
\partial g(y^{k+1})&\ni B^{\top}p^{k}+\alpha B^{\top}(Ax^{k}-By^{k+1})\\
&\quad+\nabla\psi(y^{k})-\nabla\psi(y^{k+1})\\
&= B^{\top}p^{k+1}+\alpha
B^{\top}(Ax^{k}-Ax^{k+1})\\
&\quad+\nabla\psi(y^{k})-\nabla\psi(y^{k+1}).
\end{align*}
Letting $j\to\infty$ in the above formulas yields
\begin{align*}
A^{\top}p^*=-\nabla f(x^*), B^{\top}p^*\in \partial g(y^{*}),
Ax^*=By^*,
\end{align*}
which implies that $z^*$ is a stationary point.
\end{proof}
%%%%%%%%%%%%%%%%%%%%%%%%%%%%%%%%%%%%%%%%%%%%%%%%%%%%%%%%%%%%%%%%%%%%%%%%%%

\begin{lemma}\label{E3}
Let $\hat{z}^{k+1}:=(x^{k+1},y^{k+1},p^{k+1},x^{k})$. Then there
exists $\kappa>0$ such that for each $k$
\begin{align*}
\mathrm{dist}(0,\partial
\hat{L}(\hat{z}^{k+1}))&\le\kappa(\|x^k-x^{k+1}\|+\|x^k-x^{k-1}\|\\
&\quad+\|x^{k-1}-x^{k-2}\|).
\end{align*}
\end{lemma}
\begin{proof}
By the definitions of $\hat{L}$ and algorithm \eqref{A3}-\eqref{A4},
we have
\begin{align*}
 \partial\hat{L}_x(\hat{z}^{k+1})&=\nabla f(x^{k+1})+A^{\top}p^{k+1}+\sigma_0(x^{k+1}-x^k)\\
 &\quad+\alpha A^{\top}(Ax^{k+1}-By^{k+1})\\
 &=\nabla\phi(x^k)-\nabla\phi(x^{k+1})+\sigma_0(x^{k+1}-x^k)\\
 &\quad+\alpha A^{\top}(Ax^{k+1}-By^{k+1})\\
  &=\nabla\phi(x^k)-\nabla\phi(x^{k+1})+\sigma_0(x^{k+1}-x^k)\\
 &\quad+A^{\top}(p^{k+1}-p^k),
\end{align*}
where the last equality follows from \eqref{A4}. On the other hand,
it follows from \eqref{A3} that
\begin{align*}
0&\in \partial g(y^{k+1})-B^{\top}p^{k}-\alpha B^{\top}(Ax^{k}-By^{k+1})\\
 &\quad+\nabla\psi(y^{k+1})-\nabla\psi(y^{k})\\
 &= \partial g(y^{k+1})-B^{\top}p^{k+1}-\alpha B^{\top}(Ax^{k}-Ax^{k+1})\\
 &\quad+\nabla\psi(y^{k+1})-\nabla\psi(y^{k}),
\end{align*}
which implies
\begin{align*}
 &\quad \ \partial\hat{L}_y(\hat{z}^{k+1})\\
 &=\partial g(y^{k+1})-B^{\top}p^{k+1}+\alpha B^{\top}(By^{k+1}-Ax^{k+1})\\
 &\ni\nabla\psi(y^{k})-\nabla\psi(y^{k+1})+\alpha B^{\top}(Ax^{k}-Ax^{k+1})\\
 &\quad-\alpha B^{\top}(Ax^{k+1}-By^{k+1})\\
  &= \nabla\psi(y^{k})-\nabla\psi(y^{k+1})+\alpha B^{\top}(Ax^{k}-Ax^{k+1})\\
 &\quad+ B^{\top}(p^k-p^{k+1}).
\end{align*}
Also it is clear that
$\partial\hat{L}_{\hat{x}}(\hat{z}^{k+1})=-\sigma_0(x^{k+1}-x^k)$
and
\begin{align*}
  \partial\hat{L}_p(\hat{z}^{k+1})=Ax^{k+1}-By^{k+1}=\frac1\alpha(p^{k+1}-p^k).
\end{align*}
Consequently, there exists $\kappa_0>0$ so that
\begin{align*}
\mathrm{dist}(0,\partial \hat{L}(\hat{z}^{k+1}))\le
\kappa_0(\|x^k-x^{k+1}\|+ \|y^{k+1}-y^k\|+\|p^{k+1}-p^k\|).
\end{align*}
On the other hand, it follows from \eqref{D1} that
\begin{align}\label{D6}
\|p^{k+1}-p^{k}\|&\le\bigg[\frac{2(\ell_{f}+\ell_{\phi})^2}{\mu_0}
\|x^{k+1}-x^{k}\|^2 \nonumber\\
&\quad+\frac{2\ell_{\phi}^2}{\mu_0}\|x^{k}-x^{k-1}\|^2\bigg]^{1/2}\nonumber\\
&\le\frac{\sqrt{2}(\ell_{f}+\ell_{\phi})}{\sqrt{\mu_0}}
\|x^{k+1}-x^{k}\|\nonumber\\
&\quad +\frac{\sqrt{2}\ell_{\phi}}{\sqrt{\mu_0}}\|x^{k}-x^{k-1}\|\nonumber\\
&\le\frac{\sqrt{2}(\ell_{f}+\ell_{\phi})}{\sqrt{\mu_0}}
(\|x^{k+1}-x^{k}\|+\|x^{k}-x^{k-1}\|)\nonumber\\
&=\kappa_1(\|x^{k+1}-x^{k}\|+\|x^{k}-x^{k-1}\|),
\end{align}
where we have defined
$\kappa_1:=\sqrt{2}(\ell_{f}+\ell_{\phi})/\sqrt{\mu_0}.$ Furthermore,
it follows from \eqref{B3} that
\begin{align}
\|y^k-y^{k+1}\|&\le\frac{\sqrt{2}}{\alpha\sqrt{\mu_B}}(\|p^k-p^{k+1}\|^2+\|p^k-p^{k-1}\|^2\nonumber\\
&\quad +\alpha^2\|A\|^2\|x^{k+1}-x^{k}\|^2)^{1/2}\nonumber\\
&\le\frac{\sqrt{2}}{\alpha\sqrt{\mu_B}}(\|p^k-p^{k+1}\|+\|p^k-p^{k-1}\|\nonumber\\
&\quad+\alpha\|A\|\|x^{k+1}-x^{k}\|)\nonumber\\
&\le\frac{\sqrt{2}}{\alpha\sqrt{\mu_B}}((\kappa_1+\alpha\|A\|)\|x^k-x^{k+1}\|\nonumber\\
&\quad  +2\kappa_1\|x^k-x^{k-1}\|+\kappa_1\|x^{k-1}-x^{k-2}\|)\nonumber\\
&=\kappa_2(\|x^k-x^{k+1}\|+\|x^k-x^{k-1}\|\nonumber\\
&\quad+\|x^{k-1}-x^{k-2}\|)\label{D7},
\end{align}
where we have defined
$\kappa_2:=\sqrt{2}(2\kappa_1+\alpha\|A\|)/\alpha\sqrt{\mu_B}.$
Hence, with $\kappa:=\kappa_0(\kappa_1+\kappa_2)$, we immediately
obtain the inequality as desired.
\end{proof}

%-------------------------------------------------------------------------------------------

\begin{theorem}\label{T1}
Let Assumption \ref{j1} be fulfilled. If $z^k:=(x^k,y^k,p^k)$ is
bounded, then
$$\sum_{k=0}^{\infty}\|z^k-z^{k+1}\|<\infty.$$
Moreover  the sequence $(z^k)$ converges to a stationary point of
problem \eqref{p1}.
\end{theorem}

\begin{proof}
Let $\hat{z}^{k+1}=(x^{k+1},y^{k+1},p^{k+1},x^{k})$ and let $\Omega$
denote the cluster point set of $\hat{z}^k$. By Lemma \ref{E2}, the
sequence $x^k$ is asymptotically  regular, then the sequence $x^k$
and $x^{k+1}$ share the the same cluster points. Hence we can take
$\hat{z}^*:=(x^{*},y^{*},q^{*},x^{*})\in\Omega$
 and let $\hat{z}^{k_j}$ be a subsequence
of $\hat{z}^k$  converging to $\hat{z}^*$. By our hypothesis on $g$,
we have that
 $\hat{L}(\hat{z}^{k_j})\to\hat{L}(\hat{z}^{*})$.
Since by Lemma \ref{E2} the sequence $\hat{L}(\hat{z}^{k})$ is
convergent, this implies that
$\hat{L}(\hat{z}^{k})\to\hat{L}(\hat{z}^{*});$ hence the function
$\hat{L}(\cdot)$ is a constant on $\Omega$.

Let us now consider two possible cases on $\hat{L}(\hat{z}^{k})$.
First assume that there exists $k_0\in\mathbb{N}$ such that
$\hat{L}_{k_0}=\hat{L}(\hat{z}^{*}).$ Then we  deduce from Lemma
\ref{E1} that for any $k>k_0$
 \begin{align*}
\sigma_1\|x^{k+1}-x^k\|^2&\le
\hat{L}(\hat{z}^{k})-\hat{L}(\hat{z}^{k+1})\\
&\le \hat{L}(\hat{z}^{k_0})-\hat{L}(\hat{z}^{*})=0,
 \end{align*}
where we have used the fact that $\hat{L}(\hat{z}^{k})$ is
nonincreasing. This together with \eqref{D6} and \eqref{D7} implies
that $z^k$ is a constant sequence except for some finite terms, and
thus it is a convergent sequence.

Let us now assume that $\hat{L}(\hat{z}^{k})>\hat{L}(\hat{z}^{*})$
for each $k\in\mathbb{N}$. By our hypothesis on $f$ and $g$, it is
clear that $\hat{L}(\cdot)$ is a sub-analytic function and thus
satisfies the K-L inequality. Thus by Lemma \ref{KL} there exists
$\eta>0, \delta>0, \varphi\in\mathscr{A}_{\eta}$, such that for all
$\hat{z}$ satisfying $\mathrm{dist}(\hat{z},\Omega)<\delta$ and
$\hat{L}(\hat{z}^{*})<\hat{L}(\hat{z})<\hat{L}(\hat{z}^{*})+\eta$,
 there holds the inequality
\begin{align*}
 \varphi'(\hat{L}(\hat{z})-\hat{L}(\hat{z}^{*}))\mathrm{dist}(0, \partial \hat{L}(\hat{z}))\ge1.
\end{align*}
By the definition of $\Omega$, we have that
$\lim_k\mathrm{dist}(\hat{z}^k, \Omega)=0.$ This together with the
fact that $\hat{L}(\hat{z}^{k})\to\hat{L}(\hat{z}^*)$ implies that
there exists $k_1\in\mathbb{N}$ such that $\mathrm{dist}(\hat{z}^k,
\Omega)<\delta$ and $\hat{L}(\hat{z}^{k})<\hat{L}(\hat{z}^*)+\eta$
for all $k\ge k_1.$

In what follows let us fix  $k> k_1.$  It then follows that
$$\hat{z}^k\in\{\hat{z} :\mathrm{dist}(\hat{z},\Omega)<\delta)\}\cap\{\hat{z} :
\hat{L}(\hat{z}^{*})<\hat{L}(\hat{z})<\hat{L}(\hat{z}^{*})+\eta\}.$$
Hence
$\mathrm{dist}(0,\partial\hat{L}(\hat{z}^k))\varphi\prime(\hat{L}(\hat{z}^{k})-\hat{L}(\hat{z}^*))\ge1,$
which with Lemma \ref{E3} yields
\begin{align*}
\frac{1}{\varphi\prime(\hat{L}(\hat{z}^{k})-\hat{L}(\hat{z}^*))}
\le\kappa(\|x^k-x^{k-1}\|+\|x^{k-1}-x^{k-2}\|
+\|x^{k-2}-x^{k-3}\|).
\end{align*}
By the concavity of $\varphi$, this further implies
\begin{align*}
&\quad \ \hat{L}(\hat{z}^{k})-\hat{L}(\hat{z}^{k+1})\\
&=(\hat{L}(\hat{z}^{k})-\hat{L}(\hat{z}^*))-(\hat{L}(\hat{z}^{k+1})-\hat{L}(\hat{z}^*))\\
&\le\frac{\varphi(\hat{L}(\hat{z}^{k})-\hat{L}(\hat{z}^*))-\varphi(\hat{L}(\hat{z}^{k+1})-\hat{L}(\hat{z}^*))}{\varphi\prime(\hat{L}(\hat{z}^{k})-\hat{L}(\hat{z}^*))}\\
&\le\kappa(\|x^k-x^{k-1}\|+\|x^{k-1}-x^{k-2}\|+\|x^{k-2}-x^{k-3}\|)\\
&\quad\times[\varphi(\hat{L}(\hat{z}^{k})-\hat{L}(\hat{z}^*))-\varphi(\hat{L}(\hat{z}^{k+1})-\hat{L}(\hat{z}^*))].
\end{align*}
Hence we deduce from Lemma \ref{E1} that
\begin{align*}
&\quad\|x^{k+1}-x^k\|^2\\
&\le\frac{\kappa}{\sigma_1}(\|x^k-x^{k-1}\|+\|x^{k-1}-x^{k-2}\|+\|x^{k-2}-x^{k-3}\|)\\
&\quad\times[\varphi(\hat{L}(\hat{z}^{k})-\hat{L}(\hat{z}^*))-\varphi(\hat{L}(\hat{z}^{k+1})-\hat{L}(\hat{z}^*))],
\end{align*}
which is equivalent to
\begin{align*}
&\quad4\|x^k-x^{k+1}\|\\
&\le2(\|x^k-x^{k-1}\|+\|x^{k-1}-x^{k-2}\|+\|x^{k-2}-x^{k-3}\|)^{1/2}\\
&\quad\times2\sqrt{\frac{\kappa}{\sigma_1}}[\varphi(\hat{L}(\hat{z}^{k})-\hat{L}(\hat{z}^*))-\varphi(\hat{L}(\hat{z}^{k+1})-\hat{L}(\hat{z}^*))]^{1/2}.
\end{align*}
On the other hand, using the inequality $2ab\le a^2+b^2$, we get
\begin{align*}
&\quad2(\|x^k-x^{k-1}\|+\|x^{k-1}-x^{k-2}\|+\|x^{k-2}-x^{k-3}\|)^{1/2}\\
&\quad\times2\sqrt{\frac{\kappa}{\sigma_1}}[\varphi(\hat{L}(\hat{z}^{k})-\hat{L}(\hat{z}^*))-\varphi(\hat{L}(\hat{z}^{k+1})-\hat{L}(\hat{z}^*))]^{1/2}\\
&\le\|x^k-x^{k-1}\|+\|x^{k-1}-x^{k-2}\|+\|x^{k-2}-x^{k-3}\|\\
&\quad+4\frac{\kappa}{\sigma_1}[\varphi(\hat{L}(\hat{z}^{k})-\hat{L}(\hat{z}^*))-\varphi(\hat{L}(\hat{z}^{k+1})-\hat{L}(\hat{z}^*))],
\end{align*}
so that
\begin{align*}
&\quad4\|x^k-x^{k+1}\|\\
&\le\|x^k-x^{k-1}\|+\|x^{k-1}-x^{k-2}\|+\|x^{k-2}-x^{k-3}\|\\
&\quad+4\frac{\kappa}{\sigma_1}[\varphi(\hat{L}(\hat{z}^{k})-\hat{L}(\hat{z}^*))-\varphi(\hat{L}(\hat{z}^{k+1})-\hat{L}(\hat{z}^*))].
\end{align*}
Consequently we have
\begin{align*}
&\quad \sum_{i=k_1}^{k}4\|x^i-x^{i+1}\|\\
&\le\sum_{i=k_1}^{k}(\|x^i-x^{i-1}\|+\|x^{i-1}-x^{i-2}\|+\|x^{i-2}-x^{i-3}\|)\\
&\quad+4\frac{\kappa}{\sigma_1}\sum_{i=k_1}^{k}
[\varphi(\hat{L}(\hat{z}^i)-\hat{L}(\hat{z}^*))-\varphi(\hat{L}(\hat{z}^{i+1})-\hat{L}(\hat{z}^*))],
\end{align*}
which is equivalent to
\begin{align*}
&\quad\sum_{i=k_1}^{k}\|x^i-x^{i+1}\|\\
&\le\sum_{i=k_1}^{k}(\|x^i-x^{i-1}\|-\|x^i-x^{i+1}\|)\\
&\quad+\sum_{i=k_1}^{k}(\|x^{i-1}-x^{i-2}\|-\|x^i-x^{i+1}\|)\\
&\quad+\sum_{i=k_1}^{k}(\|x^{i-2}-x^{i-3}\|-\|x^i-x^{i+1}\|)\\
&\quad+4\frac{\kappa}{\sigma_1}\sum_{i=k_1}^{k}
[\varphi(\hat{L}(\hat{z}^i)
-\hat{L}(\hat{z}^*))-\varphi(\hat{L}(\hat{z}^{i+1})-\hat{L}(\hat{z}^*))]\\
&\le3\|x^{k_1}-x^{k_1-1}\|+2\|x^{k_1-1}-x^{k_1-2}\|+\|x^{k_1-2}-x^{k_1-3}\|\\
&\quad+4\frac{\kappa}{\sigma_1}[\varphi(\hat{L}(\hat{z}^{k_1})-\hat{L}(\hat{z}^*))-\varphi(\hat{L}(\hat{z}^{k+1})-\hat{L}(\hat{z}^*))]\\
&\le3\|x^{k_1}-x^{k_1-1}\|+2\|x^{k_1-1}-x^{k_1-2}\|+\|x^{k_1-2}-x^{k_1-3}\|\\
&\quad+4\frac{\kappa}{\sigma_1}\varphi(\hat{L}(\hat{z}^{k_1})-\hat{L}(\hat{z}^*)),
\end{align*}
where the last inequality follows from the fact that
$\varphi(\hat{L}(\hat{z}^{k+1})-\hat{L}(\hat{z}^*))\ge0.$
 Since $k$ is chosen arbitrarily, we deduce that
$\sum_{k=0}^{\infty}\|x^k-x^{k+1}\|<\infty.$ It follows from  the
previous lemma that
\begin{align*}
\|q^{k+1}-q^{k}\|&\le\kappa_1(\|x^{k+1}-x^{k}\|+\|x^{k}-x^{k-1}\|\\
&\quad +\|y^{k+1}-y^k\|),\\
\|y^k-y^{k+1}\|&\le\kappa_2(\|x^k-x^{k+1}\|+\|x^k-x^{k-1}\|\\
&\quad +\|x^{k-1}-x^{k-2}\|).
\end{align*}
Hence $\sum_{k=0}^{\infty}(\|y^k-y^{k+1}\|+\|q^k-q^{k+1}\|)<\infty.$
Moreover we note that
\begin{align*}
\|z^k-z^{k+1}\|&=(\|x^k-x^{k+1}\|^2+\|y^k-y^{k+1}\|^2\\
&\quad+\|q^{k+1}-q^{k}\|^2)^{1/2}\\
&\le \|x^k-x^{k+1}\| +\|y^k-y^{k+1}\|\\
&\quad+\|q^{k+1}-q^{k}\|,
\end{align*}
so that we can conclude $\sum_{k=0}^{\infty}\|z^k-z^{k+1}\|<\infty.$
Consequently $(z^k)$ is a Cauchy sequence and thus is convergent,
which together with Lemma \ref{E2} completes the proof.
\end{proof}

\begin{remark}
We can deduce from \eqref{A1} that $p^k$ is bounded if $x^k$ is. So
in the above theorem, it suffices to assume that the primal
variables $x^k$ and $y^k$ are bounded, which can be automatically
fulfilled in many particular cases.
 For example, the boundedness
of $x^k$ or $y^k$ can be obtained by assuming the coerciveness of
$f$ or $g$.

\end{remark}

%%%%%%%%%%%%%%%%%%%%%%%%%%%%%%%%%%%%%%%%%%%%%%%%%%%%%%%%%%%%%%%%%%%%%%%%%%%%%%%%%%%%%%
\subsection{The case that $B$  is not injective}

\begin{lemma}\label{L1}
Let Assumption \ref{j2} be fulfilled. For each $k\in\mathbb{N}$
there exists $\sigma_i>0, i=0,1$ such that
\begin{align*}
&\sigma_1(\|x^{k+1}-x^k\|^2+\|y^{k+1}-y^k\|^2)\\
&\le\tilde{L}(x^{k},y^{k},p^{k},
x^{k-1})-\tilde{L}(x^{k+1},y^{k+1},p^{k+1},x^k),
\end{align*}
 where
$\tilde{L}(x,y,p,\tilde{x}):=L_\alpha(x,y,p)+\frac{\sigma_0}{2}
\|x-\tilde{x}\|^2.$
\end{lemma}

\begin{proof}
Since $\psi$ is strongly convex, we have
\begin{align*}
L_{\alpha}(x^{k},y^{k+1},p^{k})&\le L_{\alpha}(x^{k},y^{k},p^{k})-\Delta_{\psi}(y^{k+1},y^k)\\
&\le L_{\alpha}(x^{k},y^{k},p^{k})-\frac{\mu_2}{2}\|y^{k+1}-y^k\|^2,
\end{align*}
which implies
\begin{align*}
&L_{\alpha}(x^{k+1},y^{k+1},p^{k})-L_{\alpha}(x^{k},y^{k},p^{k})\\
&\le-\frac{\mu_1}{2}\|x^{k+1}-x^k\|^2-\frac{\mu_2}{2}\|y^{k+1}-y^k\|^2.
\end{align*}
Moreover we deduce form \eqref{D1} and \eqref{A4} that
\begin{align*}
&L_{\alpha}(x^{k+1},y^{k+1},p^{k+1})-L_{\alpha}(x^{k},y^{k},p^{k})\\
&=L_{\alpha}(x^{k+1},y^{k+1},p^{k+1})-L_{\alpha}(x^{k+1},y^{k+1},p^{k})\\
&\quad +L_{\alpha}(x^{k+1},y^{k+1},p^{k})-L_{\alpha}(x^{k},y^{k},p^{k})\\
&\le -\frac{\mu_1}{2}\|x^{k+1}-x^k\|^2
-\frac{\mu_2}{2}\|y^{k+1}-y^k\|^2\\
&\quad+\frac1\alpha\|p^{k+1}-p^{k}\|^2\\
&\le -\frac{\mu_1}{2}\|x^{k+1}-x^k\|^2
-\frac{\mu_2}{2}\|y^{k+1}-y^k\|^2\\
&\quad+\frac{2(\ell_{f}+\ell_{\phi})^2}{\alpha\mu_0}
\|x^{k+1}-x^{k}\|^2\\
&\quad +\frac{2\ell_{\phi}^2}{\alpha\mu_0}\|x^{k}-x^{k-1}\|^2,
\end{align*}
which is equivalent to
\begin{align*}
&L_{\alpha}(x^{k+1},y^{k+1},p^{k+1})+\frac{2\ell_{\phi}^2}{\alpha\mu_0}\|x^{k+1}-x^k\|^2\\
&\le
L_{\alpha}(x^{k},y^{k},p^{k})+\frac{2\ell_{\phi}^2}{\alpha\mu_0}\|x^{k}-x^{k-1}\|^2\\
&\quad-\frac{\mu_2}{2}\|y^{k+1}-y^k\|^2\\
&\quad-\left(\frac{\mu_1}{2}-\frac{2(\ell_{f}+\ell_{\phi})^2}{\alpha\mu_0}-
  \frac{2\ell_{\phi}^2}{\alpha\mu_0}\right)\|x^{k}-x^{k+1}\|^2.
\end{align*}
Let  us now define
\begin{align*}
&\sigma_0=\frac{4\ell_{\phi}^2}{\alpha\mu_0}, \ \sigma_1=
\min\left(\frac{\mu_2}{2},\frac{\mu_1}{2}-\frac{2(\ell_{f}+\ell_{\phi})^2}{\alpha\mu_0}-
  \frac{2\ell_{\phi}^2}{\alpha\mu_0}\right).
\end{align*}
Clearly both  $\sigma_i$ are positive and thus the desired
inequality follows.
\end{proof}

\begin{lemma}\label{L2}
If the sequence $z^k:=(x^k,y^k,p^k)$ is bounded, then
$$\sum_{k=0}^{\infty}\|z^k-z^{k+1}\|^2<\infty.$$
 In particular the
sequence $\|z^k-z^{k+1}\|$ is asymptotically regular, namely
$\|z^k-z^{k+1}\|\to0$ as $k\to\infty.$ Moreover any cluster point
 of $z^k$ is a stationary
point of $L_\alpha.$
\end{lemma}

\begin{proof}
Analogously, we can deduce as in Lemma \ref{E2} that the sequence
$\tilde{L}(\tilde{z}^{k})$ is convergent and
$\tilde{L}(\tilde{z}^{k})\ge
 \tilde{L}(\tilde{z}^{*})$ for each $k$, where $\tilde{z}^{k}:=(x^k,y^k,p^k,
 x^{k-1})$ and $\tilde{L}$ is defined as in Lemma \ref{L1}.
Now fix any $k\in\mathbb{N}$. It then follows from Lemma \ref{L1}
that
\begin{align*}
&\sigma_1\sum_{i=0}^{k}(\|x^i-x^{i+1}\|^2+\|y^i-y^{i+1}\|^2)\\
&\le\sum_{i=0}^{k}(\tilde{L}(\tilde{z}^i)-\tilde{L}(\tilde{z}^{i+1})=\tilde{L}(\tilde{z}^{0})-\tilde{L}(\tilde{z}^{k+1})\\
&\le\tilde{L}(\tilde{z}^{0})-\tilde{L}(\tilde{z}^*)<\infty.
\end{align*}
Since $k$ is chosen arbitrarily, we can deduce that
$\sum_{k=0}^{\infty}(\|x^k-x^{k+1}\|^2+\|y^k-y^{k+1}\|^2)<\infty,$
which with \eqref{D1} implies $\sum_k\|p^k-p^{k+1}\|^2<\infty$, so
that $\sum_{k=0}^{\infty}\|z^k-z^{k+1}\|^2<\infty;$ in particular
$\|z^k-z^{k+1}\|\to0$. It is clear  that any cluster point of $z^k$ is
a stationary point of function $L_\alpha$.
\end{proof}

%%%%%%%%%%%%%%%%%%%%%%%%%%%%%%%%%%%%%%%%%%%%%%%%%%%%%%%%%%%%%%%%%%%%%%%%%%
The proof of the following lemma  is similar to  that of Lemma
\ref{E3}, so we omit the details.

\begin{lemma}\label{L3}
Let $\tilde{z}^{k+1}=(x^{k+1},y^{k+1},p^{k+1},x^{k})$. Then for each
$k$ there exists $\kappa>0$ such that
\begin{align*}
\mathrm{dist}(0,\partial
\tilde{L}(\tilde{z}^{k+1}))&\le\kappa(\|x^k-x^{k+1}\|+\|y^k-y^{k+1}\|+\|x^k-x^{k-1}\|).
\end{align*}
\end{lemma}

%-------------------------------------------------------------------------------------------

\begin{theorem}\label{T2}
Assume that Assumption \ref{j2} is fulfilled. If the sequence
$z^k:=(x^k,y^k,q^k)$ is bounded, then
$$\sum_{k=0}^{\infty}\|z^k-z^{k+1}\|<\infty.$$
In particular the sequence $(z^k)$ converges to a stationary point
of $L_{\alpha}$.
\end{theorem}

\begin{proof}
Let $\tilde{z}^{k+1}=(x^{k+1},y^{k+1},q^{k+1},x^{k})$ and let
$\Omega$ be  the cluster point  set of $\tilde{z}^k$. Similar to the
proof of
 Theorem \ref{T1}, we can find  a sufficient large $k_1$ such that for all $k>k_1$
 $$\tilde{z}^k\in\{\tilde{z} :\mathrm{dist}(\tilde{z},\Omega)<\delta \}\cap\{\tilde{z} :
\tilde{L}(\tilde{z}^{*})<\tilde{L}(\tilde{z})<\tilde{L}(\tilde{z}^{*})+\eta\}.$$

In what follows, let us fix  $k>k_1$. Then the K-L inequality
$$\mathrm{dist}(0,\partial\tilde{L}(\tilde{z}^k))\varphi\prime(\tilde{L}(\tilde{z}^{k})-\tilde{L}(\tilde{z}^*))\ge1$$
together with Lemma \ref{L3} implies
\begin{align*}
&\frac{1}{\varphi\prime(\tilde{L}(\tilde{z}^{k})-\tilde{L}(\tilde{z}^*))}\\
&\le\kappa(\|x^k-x^{k-1}\|+\|y^k-y^{k-1}\|+\|x^{k-2}-x^{k-1}\|),
\end{align*}
so that the concavity of $\varphi$ yields
\begin{align*}
&\quad\tilde{L}(\tilde{z}^{k})-\tilde{L}(\tilde{z}^{k+1})\\
&=(\tilde{L}(\tilde{z}^{k})-\tilde{L}(\tilde{z}^*))-(\tilde{L}(\tilde{z}^{k+1})-\tilde{L}(\tilde{z}^*))\\
&\le\frac{\varphi(\tilde{L}(\tilde{z}^{k})-\tilde{L}(\tilde{z}^*))-\varphi(\tilde{L}(\tilde{z}^{k+1})-\tilde{L}(\tilde{z}^*))}{\varphi\prime(\tilde{L}(\tilde{z}^{k})-\tilde{L}(\tilde{z}^*))}\\
&\le\kappa(\|x^k-x^{k-1}\|+\|y^k-y^{k-1}\|+\|x^{k-2}-x^{k-1}\|)\\
&\quad\times[\varphi(\tilde{L}(\tilde{z}^{k})-\tilde{L}(\tilde{z}^*))-\varphi(\tilde{L}(\tilde{z}^{k+1})-\tilde{L}(\tilde{z}^*))].
\end{align*}
From Lemma \ref{L1}, this implies
\begin{align*}
&\quad\|x^{k+1}-x^k\|^2+\|y^{k+1}-y^k\|^2\\
&\le\frac{\kappa}{\sigma_1}(\|x^k-x^{k-1}\|+\|y^k-y^{k-1}\|+\|x^{k-2}-x^{k-1}\|)\\
&\quad\times[\varphi(\tilde{L}(\tilde{z}^{k})-\tilde{L}(\tilde{z}^*))-\varphi(\tilde{L}(\tilde{z}^{k+1})-\tilde{L}(\tilde{z}^*))],
\end{align*}
which is equivalent to
\begin{align*}
&\quad3(\|x^k-x^{k+1}\|+\|y^k-y^{k+1}\|)\\
&\le3\sqrt{2}(\|x^{k+1}-x^k\|^2+\|y^{k+1}-y^k\|^2)^{1/2}\\
&\le2(\|x^k-x^{k-1}\|+\|y^k-y^{k-1}\|+\|x^{k-2}-x^{k-1}\|)^{1/2}\\
&\quad\times\sqrt{\frac{9\kappa}{2\sigma_1}}[\varphi(\tilde{L}(\tilde{z}^{k})
-\tilde{L}(\tilde{z}^*))-\varphi(\tilde{L}(\tilde{z}^{k+1})-\tilde{L}(\tilde{z}^*))]^{1/2}.
\end{align*}
It is readily seen that
\begin{align*}
&\quad \ 2(\|x^k-x^{k-1}\|+\|y^k-y^{k-1}\|+\|x^{k-2}-x^{k-1}\|)^{1/2}\\
&\quad\times\sqrt{\frac{9\kappa}{2\sigma_1}}[\varphi(\tilde{L}(\tilde{z}^{k})-\tilde{L}(\tilde{z}^*))-\varphi(\tilde{L}(\tilde{z}^{k+1})-\tilde{L}(\tilde{z}^*))]^{1/2}\\
&\le \|x^k-x^{k-1}\|+\|y^k-y^{k-1}\|+\|x^{k-2}-x^{k-1}\|\\
&\quad+\frac{9\kappa}{2\sigma_1}[\varphi(\tilde{L}(\tilde{z}^{k})-\tilde{L}(\tilde{z}^*))-\varphi(\tilde{L}(\tilde{z}^{k+1})-\tilde{L}(\tilde{z}^*))],
\end{align*}
so that
\begin{align*}
&\quad \ 3(\|x^k-x^{k+1}\|+\|y^k-y^{k+1}\|)\\
&\le \|x^k-x^{k-1}\|+\|y^k-y^{k-1}\|+\|x^{k-2}-x^{k-1}\|\\
&\quad+\frac{9\kappa}{2\sigma_1}[\varphi(\tilde{L}(\tilde{z}^{k})-\tilde{L}(\tilde{z}^*))-\varphi(\tilde{L}(\tilde{z}^{k+1})-\tilde{L}(\tilde{z}^*))].
\end{align*}
Hence we have
\begin{align*}
&\quad \ \sum_{i=k_1}^{k}3(\|x^i-x^{i+1}\|+\|y^i-y^{i+1}\|)\\
&\le\sum_{i=k_1}^{k}(\|x^i-x^{i-1}\|+\|y^i-y^{i-1}\|+\|x^{i-1}-x^{i-2}\|)\\
&\quad+\frac{9\kappa}{2\sigma_1}\sum_{i=k_1}^{k}
[\varphi(\tilde{L}(\tilde{z}^i)-\tilde{L}(\tilde{z}^*))-\varphi(\tilde{L}(\tilde{z}^{i+1})-\tilde{L}(\tilde{z}^*))],
\end{align*}
from which it follows that
\begin{align*}
&\quad \ \sum_{i=k_1}^{k}\|x^i-x^{i+1}\|+2\sum_{i=k_1}^{k}\|y^i-y^{i+1}\|\\
&\le\sum_{i=k_1}^{k}(\|x^i-x^{i-1}\|-\|x^i-x^{i+1}\|)\\
&\quad +\sum_{i=k_1}^{k}(\|x^{i-1}-x^{i-2}\|-\|x^i-x^{i+1}\|)\\
&\quad+\sum_{i=k_1}^{k}(\|y^i-y^{i-1}\|-\|y^i-y^{i+1}\|)\\
&\quad +\frac{9\kappa}{2\sigma_1}\sum_{i=k_1}^{k}
[\varphi(\tilde{L}(\tilde{z}^i)
-\tilde{L}(\tilde{z}^*))-\varphi(\tilde{L}(\tilde{z}^{i+1})-\tilde{L}(\tilde{z}^*))]\\
&=\|x^{k_1-1}-x^{k_1-2}\|+2\|x^{k_1}-x^{k_1-1}\|\\
&\quad -\|x^{k-1}-x^{k_1-2}\|-2\|x^{k}-x^{k+1}\|\\
&\quad+\|y^{k_1-1}-y^{k_1}\|-\|y^{k}-y^{k+1}\|\\
&\quad +\frac{9\kappa}{2\sigma_1}
[\varphi(\tilde{L}(\tilde{z}^{k_1})-\tilde{L}(\tilde{z}^*))-\varphi(\tilde{L}(\tilde{z}^{k+1})-\tilde{L}(\tilde{z}^*))]\\
&\le\|x^{k_1-1}-x^{k_1-2}\|+2\|x^{k_1}-x^{k_1-1}\|+\|y^{k_1}-y^{k_1-1}\|\\
&\quad
+\frac{9\kappa}{2\sigma_1}\varphi(\tilde{L}(\tilde{z}^{k_1})-\tilde{L}(\tilde{z}^*)),
\end{align*}
where the last inequality follows from the fact that
$\varphi(\tilde{L}(\tilde{z}^{k_1})-\tilde{L}(\tilde{z}^*))\ge0.$
 Since $k$ is chosen arbitrarily, we can deduce that
$\sum_{k=0}^{\infty}(\|x^k-x^{k+1}\|+\|y^k-y^{k+1}\|)<\infty,$ which
together with \eqref{D6} enables us to deduce that
 $\sum_{k=0}^{\infty}\|q^k-q^{k+1}\|<\infty,$ and moreover
 $\sum_{k=0}^{\infty}\|z^k-z^{k+1}\|<\infty.$ Consequently
 $(z^k)$ is convergent, which together with
Lemma \ref{L1} completes the proof.
\end{proof}

\section{A demonstration example}

In compressed sensing, a fundamental problem is recovering an
$n$-dimensional sparse signal $x$ from a set of $m$ incomplete
measurements with $m<<n$. It is possible as long as the number of
nonzero elements of $x$ is small enough. In such case one needs to
find the sparsest solution of a linear system, which can be modeled
as
\begin{align*}
\min_{x\in\mathbb{R}^n}& \ \|x\|_0 \\
 {\rm s.t.}& \ Dx=b,
\end{align*}
where $D\in\mathbb{R}^{{m}\times n}$ is the  measurement matrix,
$b\in\mathbb{R}^{m}$ is the observed data, and $\|x\|_0$ denotes the
number of nonzero elements of $x$. In most cases, the sparsity 
 is usually  demonstrated under a linear transformation, for
example in total variation denoising \cite{rof}. This then requires
to solve:
\begin{align*}
\min_{x\in\mathbb{R}^n} & \ \|Ax\|_0  \nonumber\\
 {\rm s.t.}& \ Dx=b,
\end{align*}
 or its regularization version:
\begin{align}\label{p2}
\min_{x\in\mathbb{R}^n}\|Dx-b\|^2+\lambda\|Ax\|_0,
\end{align}
where $\lambda>0$ is a regularization parameter and
$A\in\mathbb{R}^{(n-1)\times n}$ is the difference matrix, say,
defined by
\begin{align}\label{D3}
A_{ij}=\left\{
  \begin{array}{cl}
  1, & \ j=i+1  \\
  -1, & \  j=i \\
  0, & \ otherwise.
  \end{array}
\right.
\end{align}
It is clear that the difference matrix has full-row rank.

In general, the above-mentioned problems are intractable because it
is in fact a NP-hard problem. To overcome this difficulty, one may
relax the $\ell_0$ norm to the $\ell_1$ norm as in \eqref{p2}, which
then leads to a convex composite problem:
\begin{align}\label{l1}
\begin{split}
\min & \ \|Dx-b\|^2+\lambda\|y\|_{1} \\
 {\rm s.t.}& \ Ax=y.
\end{split}
\end{align}
where  $\|x\|_1=\sum_i|x_i|$ stands for the  $\ell_1$ norm. Applying
 BADMM to problem \eqref{l1} with $\phi(x)=\psi(x)=\mu\|x\|^2/2$ yields
 \begin{align}\label{sadmm}
 \begin{split}
 y^{k+1}&=H(Ax^{k}+p^k/\alpha;\lambda/\alpha)\\
x^{k+1}&=(2D^{\top}D+\alpha A^{\top}A+\mu I)^{-1}w^{k+1} \\
p^{k+1}&=p^k+\alpha(Ax^{k+1}-By^{k+1}),
 \end{split}
\end{align}
where $w^{k+1}=\mu x^k+\alpha
A^{\top}y^{k+1}+2D^{\top}b-A^{\top}p^k$ and $S(\cdot;\mu)$ is the
soft shrinkage operator.

Nevertheless, the $\ell_1$ regularization has been shown to be
suboptimal in many cases; in particular  it cannot enforce further
sparsity, since the $\ell_1$ norm is a loose approximation of the
$\ell_0$ norm and often leads to an overpenalized problem. To
overcome the drawback caused by the $\ell_1$ regularization, an
alternative way is to replace the $\ell_1$ norm by the  $\ell_{1/2}$
quasi norm in problem \eqref{p2} (see e.g. \cite{xc,zlx,zfx,zxz}).
This then leads to the following nonconvex composite problem:
\begin{align}\label{c5}
\min & \ \|Dx-b\|^2+\lambda\|y\|^{1/2}_{1/2}\nonumber\\
 {\rm s.t.}& \ Ax=y.
\end{align}
 Applying BADMM to problem \eqref{c5} also with  $\phi(x)=\psi(x)=\mu\|x\|^2/2$ yields
 \begin{align}\label{hadmm}
 \begin{split}
 y^{k+1}&=H(Ax^{k}+p^k/\alpha;2\lambda/\alpha)\\
x^{k+1}&=(2D^{\top}D+\alpha A^{\top}A+\mu I)^{-1}w^{k+1} \\
p^{k+1}&=p^k+\alpha(Ax^{k+1}-By^{k+1}).
 \end{split}
\end{align}
Here $w^{k+1}=\mu x^k+\alpha A^{\top}y^{k+1}+2D^{\top}b-A^{\top}p^k
$ and $H(\cdot;\mu)$ is the half shrinkage operator \cite{xc}
defined as $H(x;\mu)=\{h_{ \mu}(x_1),h_{ \mu}(x_2)\cdots h_{
\mu}(x_n)\}^{\top}$ with
\begin{align*}
h_{ \mu}(x_i)=\left\{
        \begin{array}{cl}
\frac{2x_i}{3}\big(1+ \cos\frac23(\pi-\varphi(|x_i|))\big),
& \  |x_i|>\frac{\sqrt[3]{54}}{4}\mu^{2/3}; \\
0, & \ h_{ \mu}(x_i)=0,
        \end{array}
      \right.
\end{align*}
with $\varphi(x)=\arccos(\frac{ \mu}{8} (\frac{|x_i|}{3})^{-3/2})$.

For simplicity, we denote algorithms \eqref{hadmm} and \eqref{sadmm}
by HADMM and SADMM, respectively. We now conduct an experiment to
verify convergence of the nonconvex BADMM, and reveal its advantages in
sparsity-inducing and efficiency through comparing the performance of 
HADMM and SADMM. In the experiment, the difference
matrix $A\in\mathbb{R}^{511\times512}$ was generated according to
\eqref{D3}, and $D\in\mathbb{R}^{256\times512}$ was randomly
generated with Gaussian $\mathcal{N}(0,1/256)$ i.i.d. entries.  We
applied the HADMM and SADMM with the same parameters $\lambda
=0.015, \alpha=10$ and $\mu_1 =\mu_2= 10$.

The experimental results are shown in Figure 1, where the
restoration accuracy is measured by means of the mean squared error
\begin{align*}
\mathrm{MSE}(\|x^*-x^k\|) &= \frac1n \|x^*-x^k\|,\\
\mathrm{MSE}(\|y^*-y^k\|) &= \frac1n \|y^*-y^k\|.
\end{align*}
Here $(x^*,y^*)$ is the true solution of the problem. As shown in
Figure 1, both sequences $x^k$ and $y^k$  were fairly near the true
solution. i.e., the convergence is justified. It is readily seen
that HADMM converges faster than SADMM does. Moreover, this
difference is particularly notable for $y^k$. This supports in
partial the advantage of the nonconvex model \eqref{c5} over the
convex model \eqref{l1} for the considered problem.

\begin{figure}
\centering
\subfigure[$\mathrm{MSE}(\|x^k-x^*\|)$] { \label{fig:a}
\includegraphics[width=0.45\columnwidth]{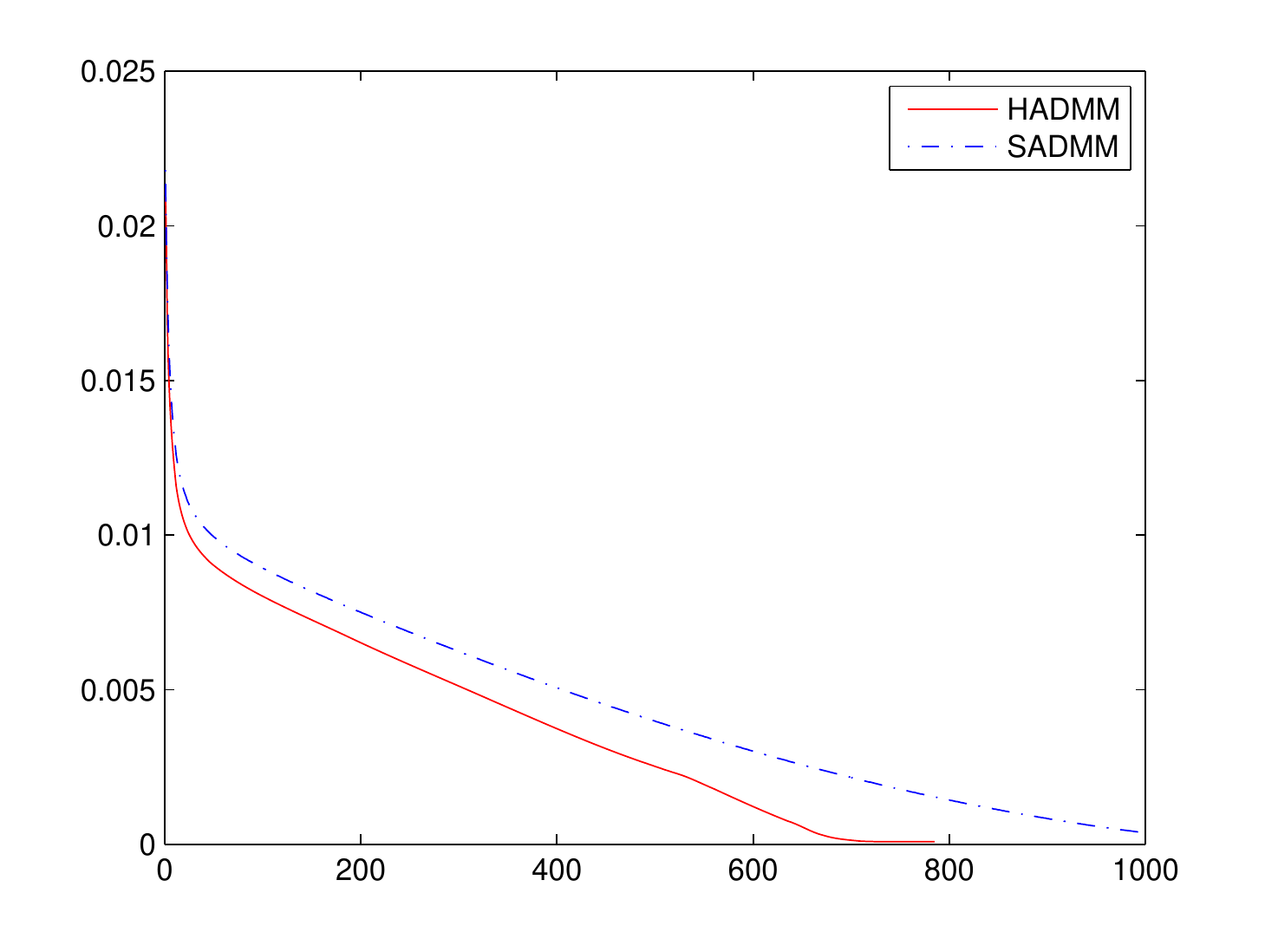}
}
\subfigure[$\mathrm{MSE}(\|y^k-y^*\|)$] { \label{fig:b}
\includegraphics[width=0.45\columnwidth]{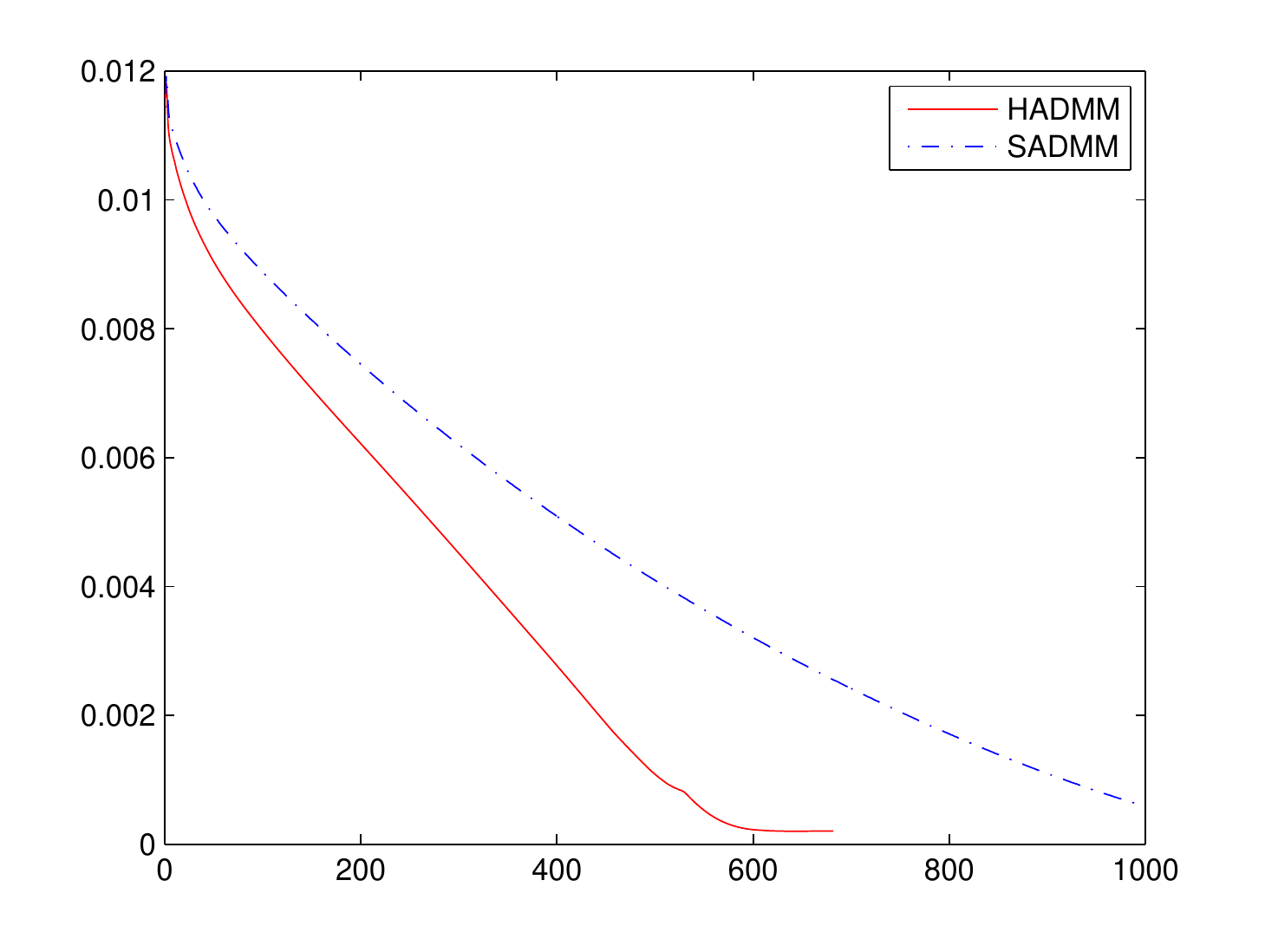}
}
\caption{Comparison the performance  of HADMM and SADMM }
\label{fig}
\end{figure}

\section{Conclusion}

In this paper, we conducted  a convergence analysis on BADMM in the
absence of convexity. We have shown that under certain conditions,  
the BADMM algorithm can converge to a stationary
point for sub-analytic functions. More importantly, our analysis is based on the
sufficient descent property of the auxiliary function, instead of the augmented
Lagrangian function.

 It is worth
noting that the order for updating the primal variables $x^k$ and
$y^k$ plays a key role in our convergence analysis. If we change the
order, namely
 first update $x^k$ and then $y^k$, this may lead  to a
difficulty to derive an relation between $x^k$ and $p^k$. Thus how
to establish the convergence results under this case is our next
subject to study.

\section*{Acknowledgement}

The first author wishes to thank   Dr. Jinshan Zeng for his valuable
suggestions on this paper. This work was partially supported by the
National 973 Programs (Grant No. 2013CB329404),  the Key Program of
National Natural Science Foundation of China (Grant No. 11131006)
and the National Natural Science Foundation of China (Grant No.
111301253).

\ifCLASSOPTIONcaptionsoff
  \newpage
\fi



%\bibliographystyle{IEEEtran}

\end{document}